\documentclass[12pt]{article}
\usepackage{amsmath,amssymb,amsthm,cite}

\newtheorem{theorem}{Theorem}[section]
\newtheorem{thmy}{Theorem}
\newtheorem{lemma}[theorem]{Lemma}

\def\barr{\begin{array}}
\def\earr{\end{array}}

\title{Finite groups with large Chermak-Delgado lattices}
\author{Georgiana Fasol\u a and Marius T\u arn\u auceanu}
\date{September 1, 2022}

\begin{document}

\maketitle

\begin{abstract}
    Given a finite group $G$, we denote by $L(G)$ the subgroup lattice of $G$ and by ${\cal CD}(G)$ the Chermak-Delgado lattice of $G$.
    In this note, we determine the finite groups $G$ such that $|{\cal CD}(G)|=|L(G)|-k$, $k=1,2$.
\end{abstract}

{\small
\noindent
{\bf MSC2020\,:} Primary 20D30; Secondary 20D60, 20D99.

\noindent
{\bf Key words\,:} Chermak-Delgado measure, Chermak-Delgado lattice, subgroup lattice, generalized quaternion $2$-group.}

\section{Introduction}

Let $G$ be a finite group and $L(G)$ be the subgroup lattice of $G$. The \textit{Chermak-Delgado measure} of a subgroup $H$ of $G$ is defined by
\begin{equation}
m_G(H)=|H||C_G(H)|.\nonumber
\end{equation}Let
\begin{equation}
m^*(G)={\rm max}\{m_G(H)\mid H\leq G\} \mbox{ and } {\cal CD}(G)=\{H\leq G\mid m_G(H)=m^*(G)\}.\nonumber
\end{equation} Then the set ${\cal CD}(G)$ forms a modular, self-dual sublattice of $L(G)$, which is called the \textit{Chermak-Delgado lattice} of $G$. It was first introduced by Chermak and Delgado \cite{4}, and revisited by Isaacs \cite{5}. In the last years there has been a growing interest in understanding this lattice (see e.g. \cite{1,2,3,6,7,8,11,12,13,14}). We recall several important properties of the Chermak-Delgado measure:
\begin{itemize}
\item[$\bullet$] if $H\leq G$ then $m_G(H)\leq m_G(C_G(H))$, and if the measures are equal then $C_G(C_G(H))=H$;
\item[$\bullet$] if $H\in {\cal CD}(G)$ then $C_G(H)\in {\cal CD}(G)$ and $C_G(C_G(H))=H$;
\item[$\bullet$] the maximal member $M$ of ${\cal CD}(G)$ is characteristic and satisfies ${\cal CD}(M)={\cal CD}(G)$;
\item[$\bullet$] the minimal member $M(G)$ of ${\cal CD}(G)$ (called the \textit{Chermak-Delgado subgroup} of $G$) is characteristic, abelian and contains $Z(G)$.
\end{itemize}

In \cite{12}, the Chermak-Delgado measure of $G$ has been seen as a function
\begin{equation}
m_G:L(G)\longrightarrow\mathbb{N}^*,\, H\mapsto m_G(H),\, \forall\, H\in L(G).\nonumber
\end{equation}If $G$ is non-trivial, then $m_G$ has at least two distinct values, or equivalently ${\cal CD}(G)\neq L(G)$ (see Corollary 3 of \cite{11}). This leads to the following natural question:

\begin{center}
\textit{How large can the lattice ${\cal CD}(G)$ be}?
\end{center}

Note that the dual problem of finding finite groups with small Chermak-Delgado lattices has been studied in \cite{6,7}.
\bigskip

Our main result is stated as follows.

\begin{theorem}
Let $G$ be a finite group. Then
\begin{itemize}
\item[{\rm a)}] $|{\cal CD}(G)|=|L(G)|-1$ if and only if $G\cong\mathbb{Z}_p$ or $G\cong Q_8$;
\item[{\rm b)}] $|{\cal CD}(G)|=|L(G)|-2$ if and only if $G\cong\mathbb{Z}_{p^2}$.
\end{itemize}
\end{theorem}

For the proof of the above theorem, we need the following well-known result (see e.g. (4.4) of \cite{10}, II).

\begin{thmy}
A finite $p$-group has a unique subgroup of order $p$ if and only if it is either cyclic or a generalized quaternion $2$-group.
\end{thmy}

We recall that a \textit{generalized quaternion $2$-group} is a group of order $2^n$ for some positive integer $n\geq 3$, defined by
\begin{equation}
Q_{2^n}=\langle a,b \mid a^{2^{n-2}}= b^2, a^{2^{n-1}}=1, b^{-1}ab=a^{-1}\rangle.\nonumber
\end{equation}

We also need the following theorem taken from \cite{12}.

\begin{thmy}
Let $G$ be a finite group. For each prime $p$ dividing the order of $G$ and $P\in{\rm Syl}_p(G)$, let $|Z(P)|=p^{n_p}$. Then
\begin{equation}
|{\rm Im}(m_G)|\geq 1+\sum_{p}n_p\,.
\end{equation}
\end{thmy}

Finally, we indicate a natural open problem concerning the above study.

\bigskip\noindent{\bf Open problem.} Determine the finite groups $G$ such that $|{\cal CD}(G)|=|L(G)|-k$, where $k\geq 3$.
\bigskip

Most of our notation is standard and will usually not be repeated here. Elementary notions and results on groups can be found in \cite{5}.
For subgroup lattice concepts we refer the reader to \cite{9}.

\section{Proof of the main result}

First of all, we solve the problem for generalized quaternion $2$-groups.

\begin{lemma}
Under the above notation, we have:
\begin{itemize}
\item[{\rm a)}] $|{\cal CD}(Q_{2^n})|=|L(Q_{2^n})|-1$ if and only if $n=3$, i.e. $G\cong Q_8$;
\item[{\rm b)}] $|{\cal CD}(Q_{2^n})|\neq |L(Q_{2^n})|-2$ for all $n\geq 3$.
\end{itemize}
\end{lemma}

\begin{proof}
We easily obtain
\begin{equation}
    m^*(Q_{2^n})=2^{2n-2}, \, \forall\, n\geq 3\nonumber
\end{equation}and
\begin{equation}
    {\cal CD}(Q_{2^n})=\left\{\barr{lll}
    \{Q_8, \langle a\rangle, \langle b\rangle, \langle ab\rangle, \langle a^2\rangle\},&n=3\\
    \{\langle a\rangle\},&n\geq 4.\earr\right.\nonumber
\end{equation}These lead immediately to the desired conclusions.
\end{proof}

We are now able to prove our main result.

\bigskip\noindent{\bf Proof of Theorem 1.1.}
\medskip

{\bf a)} Since $|{\cal CD}(G)|=|L(G)|-1$, we have ${\cal CD}(G)=L(G)\setminus\{H_0\}$, where $H_0\leq G$. We infer that $|{\rm Im}(m_G)|=2$ and so $G$ is a $p$-group with $|Z(G)|=p$ by Theorem B. Then $m_G(1)<m_G(Z(G))$, implying that
\begin{equation}
H_0=1 \mbox{ and } m^*(G)=m_G(Z(G))=p^{n+1},\nonumber
\end{equation}where $|G|=p^n$.\newpage

Assume that there exists $H\leq G$ with $|H|=p$ and $H\neq Z(G)$. Then $H\notin{\cal CD}(G)$, which shows that $H=1$, a contradiction. Thus $G$ has a unique subgroup of order $p$ and Theorem A leads to
\begin{equation}
G\cong\mathbb{Z}_{p^n} \mbox{ or } G\cong Q_{2^n} \mbox{ for some } n\geq 3.
\end{equation}In the first case we easily get $n=1$, i.e. $G\cong\mathbb{Z}_p$, while in the second one we get $G\cong Q_8$ by item a) in Lemma 2.1.
\medskip

{\bf b)} The condition $|{\cal CD}(G)|=|L(G)|-2$ means ${\cal CD}(G)=L(G)\setminus\{H_1,H_2\}$, where $H_1,H_2\leq G$. Then $|{\rm Im}(m_G)|\leq 3$. Recall that we cannot have $|{\rm Im}(m_G)|=1$.

If $|{\rm Im}(m_G)|=2$, then
\begin{equation}
m_G(H_1)=m_G(H_2)\neq m^*(G)\nonumber
\end{equation}and again $G$ is a $p$-group with $|Z(G)|=p$. It is clear that one of the two subgroups $H_1$ and $H_2$ must be trivial, say $H_1=1$. Then $G$ has at most two subgroups of order $p$, namely $Z(G)$ and possibly $H_2$. This implies that it has exactly one subgroup of order $p$ because the number of subgroups of order $p$ in a finite $p$-group is congruent to $1$ $({\rm mod}\, p)$. Consequently, one obtains again (2). For $G\cong\mathbb{Z}_{p^n}$ we easily get $n=2$, i.e. $G\cong\mathbb{Z}_{p^2}$, while for $G\cong Q_{2^n}$ we get no solution by item b) in Lemma 2.1.

If $|{\rm Im}(m_G)|=3$, then $m_G(H_1)$, $m_G(H_2)$ and $m^*(G)$ are distinct. Also, the inequality (1) becomes
\begin{equation}
3\geq 1+\sum_{p}n_p\,.\nonumber
\end{equation}Since $n_p\geq 1$ for all $p$, we have the following two possibilities:

\medskip\noindent\hspace{15mm}{\bf Case 1.} $|G|=p^n$ and $|Z(G)|\in\{p,p^2\}$\medskip

Obviously, if $G$ is abelian, we get $G\cong\mathbb{Z}_{p^2}$. Assume that $G$ is not abelian. Since $m_G(1)<m_G(Z(G))=m_G(G)$, we infer that one of the two subgroups $H_1$ and $H_2$ is trivial, and that
\begin{equation}
m^*(G)=m_G(Z(G))=m_G(G).\nonumber
\end{equation}If $|Z(G)|=p$, then $G$ has a unique subgroup of order $p$ and so it is a generalized quaternion $2$-group, contradicting item b) in Lemma 2.1. The same thing can be also said when $|Z(G)|=p^2$ because all subgroups of order $p$ of $G$ are outside of ${\cal CD}(G)$.\newpage

\medskip\noindent\hspace{15mm}{\bf Case 2.} $|G|=p^nq^m$ and the Sylow $p$-subgroups and $q$-subgroups of $G$ have centers of orders $p$ and $q$, respectively\medskip

Let $P$ be a Sylow $p$-subgroup and $Q$ be a Sylow $q$-subgroup of $G$. Since $P\subseteq C_G(Z(P))$, we have
\begin{equation}
m_G(Z(P))=p\,|C_G(Z(P))|=p^{n+1}q^x \mbox{ for some } 0\leq x\leq m,\nonumber
\end{equation}and similarly
\begin{equation}
m_G(Z(Q))=p^yq^{m+1} \mbox{ for some } 0\leq y\leq n.\nonumber
\end{equation}Also, we have
\begin{equation}
m_G(1)=p^nq^m \mbox{ and } m_G(G)=p^nq^m|Z(G)|.\nonumber
\end{equation}We observe that the measures $m_G(Z(P))$, $m_G(Z(Q))$ and $m_G(1)$ are distinct, and consequently they are all possible measures of the subgroups of $G$. We distinguish the following two subcases:

\medskip\noindent\hspace{25mm}{\bf Subcase 2.1.} $Z(G)=1$\medskip

Then $m^*(G)=m_G(1)=m_G(G)$. Indeed, if $m^*(G)=m_G(Z(P))$, then $1$, $G$ and $Z(Q)$ will be outside of ${\cal CD}(G)$, a contradiction. In the same way, we cannot have $m^*(G)=m_G(Z(Q))$. Since $m_G(P)$ is divisibly by $p^{n+1}$ and $m_G(Q)$ is divisibly by $q^{m+1}$, we infer that $m_G(P)=m_G(Z(P))$ and $m_G(Q)=m_G(Z(Q))$. Thus
\begin{equation}
P,Z(P),Q,Z(Q)\notin{\cal CD}(G)\nonumber
\end{equation}and our hypothesis implies that $P=Z(P)$ and $Q=Z(Q)$, i.e. $G$ is a non-abelian group of order $pq$. Assume that $p<q$. Then ${\cal CD}(G)$ consists of the unique subgroup of order $q$ of $G$ and therefore we obtain $|L(G)|=3$, a contradiction.

\medskip\noindent\hspace{25mm}{\bf Subcase 2.2.} $Z(G)\neq 1$\medskip

Then $m_G(1)<m_G(G)$, which shows that $m_G(G)$ equals either $m_G(Z(P))$ or $m_G(Z(Q))$. Assume that $m_G(G)=m_G(Z(P))$. Then $x=m$ and $|Z(G)|=p$, implying that
\begin{equation}
Z(G)=Z(P).
\end{equation}Note that we cannot have $m^*(G)=m_G(Z(Q))$ because in this case $1$, $Z(G)$ and $G$ will be outside of ${\cal CD}(G)$, a contradiction. Consequently, we have\newpage
\begin{equation}
m^*(G)=m_G(Z(P))=m_G(P)=m_G(Z(G))=m_G(G).\nonumber
\end{equation}It follows that $1$, $Z(Q)$ and $Q$ are not contained in ${\cal CD}(G)$, which leads to $Q=Z(Q)$. In other words, we have ${\cal CD}(G)=L(G)\setminus\{1,Q\}$. Thus ${\cal CD}(G)$ is the lattice interval
\begin{equation}
[G/Z(G)]=\{H\in L(G)\mid Z(G)\leq H\leq G\}\nonumber
\end{equation}and Corollary 2 of \cite{11} shows that $G$ is nilpotent. Then $G=P\times Q$ and so $Z(G)=Z(P)\times Q$, contradicting (3).

This completes the proof.\qed

\vspace*{3ex}
\small

\begin{minipage}[t]{7cm}
Georgiana Fasol\u a \\
Faculty of  Mathematics \\
"Al.I. Cuza" University \\
Ia\c si, Romania \\
e-mail: \!{\tt georgiana.fasola@student.uaic.ro}
\end{minipage}
\hfill\hspace{20mm}
\begin{minipage}[t]{5cm}
Marius T\u arn\u auceanu \\
Faculty of  Mathematics \\
"Al.I. Cuza" University \\
Ia\c si, Romania \\
e-mail: \!{\tt tarnauc@uaic.ro}
\end{minipage}

\end{document}